\documentclass[sn-mathphys]{sn-jnl}

\usepackage{arydshln}
\usepackage{amssymb,amsfonts,amsmath}
\usepackage{color}
\usepackage{fix-cm}

\usepackage{program}
\catcode`\|=12\relax
\usepackage{tikz}
\usepackage{pgfplots}
\usepackage{filecontents}

\theoremstyle{thmstyleone}%
\newtheorem{theorem}{Theorem}
\theoremstyle{thmstyletwo}%
\newtheorem{remark}{Remark}%

\theoremstyle{thmstylethree}%

\newcommand{\RR}{{\mathbb{R}}}
\newcommand{\NN}{{\mathbb{N}}}
\newcommand{\ZZ}{{\mathbb{Z}}}

\newcommand{\qban}{{\mathcal  L}_n}

\newcommand{\lhy}{\mathfrak{L}}
\newcommand{\lhz}{\mathfrak{L}}

\newcommand{\bet}{\mathbf{e}}

\begin{document}

\title[On computing the zeros of a class of Sobolev orthogonal polynomials]{On computing the zeros of a class of Sobolev orthogonal  polynomials}


\author[1]{\fnm{N.} \sur{Mastronardi}}\email{nicola.mastronardi@cnr.it}
\equalcont{The authors contributed equally to this work. }
\author[2]{\fnm{M.} \sur{Van Barel}}\email{marc.vanbarel@kuleuven.be}
\equalcont{The authors contributed equally to this work. }
\author[2]{\fnm{R.} \sur{Vandebril}}\email{raf.vandebril@kuleuven.be}
\equalcont{The authors contributed equally to this work. }
\author[3]{\fnm{P.} \sur{Van Dooren}}\email{paul.vandooren@uclouvain.be}
\equalcont{The authors contributed equally to this work. }

\affil*[1]{\orgdiv{Istituto per le Applicazioni del Calcolo ``M. Picone"}, \orgname{Consiglio Nazionale delle Ricerche}, \orgaddress{\street{via Amendola 122/D}, \city{Bari}, \country{Italy}}}
\affil*[2]{\orgdiv{Department of Computer Science}, \orgname{KU Leuven}, \orgaddress{\street{Celestijnenlaan 200A - Bus 2402
3001}, 
 \city{Leuven}, \country{Belgium}}}
\affil*[3]
{\orgdiv{Department of Mathematical Engineering}, \orgname{Catholic University of Louvain}, \orgaddress{\street{Avenue Georges Lemaitre 4}, \city{Louvain-la-Neuve}, \country{Belgium}}}



\abstract{
A fast and weakly stable method for computing the zeros of a particular class of hypergeometric polynomials is presented. The studied hypergeometric polynomials satisfy a higher order differential equation and generalize Laguerre  polynomials. The theoretical study of the asymptotic distribution of the spectrum of these polynomials is an active research topic. In this article we do not contribute to the theory, but provide a practical method to
contribute to further and better understanding of the asymptotic behavior.
The polynomials under consideration fit into the class of Sobolev orthogonal polynomials,  satisfying a four--term recurrence relation. This allows computing the roots via a generalized eigenvalue problem. After condition enhancing similarity transformations, the problem is transformed into the computation of the eigenvalues of a comrade matrix, which is a symmetric tridiagonal modified by a rank--one matrix.
The eigenvalues are then retrieved by relying on an existing structured rank based fast algorithm.
Numerical    examples are  reported studying the accuracy, stability  and conforming the efficiency 
for various parameter settings of the proposed approach. 
 }

\keywords{Sobolev orthogonal polynomials, zeros of polynomials,   generalized eigenvalue problem, comrade matrices}

\pacs[AMS Classification]{33C20, 65F15, 65F35}

\maketitle

\section{Introduction}
Sobolev orthogonal polynomials have been intensively studied in the last decades as a generalization of classical orthogonal polynomials.  Whereas the orthogonality of classical orthogonal polynomials is typically defined via a measure and polynomial evaluations in the discrete setting, Sobolev orthogonal polynomials also take the derivative or higher order derivatives into account when defining the inner product.
Considering Sobolev orthogonal polynomials, the overview of
Marcell\'an and Xu \cite{MarXu15} provides a recent survey of both
theoretical and application oriented developments and can serve as a
starting point for further study (see also the work of Marcell\'an,
P\'erez, and Pi\~nar \cite{MarPe96}). Van Buggenhout provides an
overview of the connection between Sobolev orthogonal polynomials and
inverse eigenvalue problems, which links to Krylov as well \cite{Bu23}.
With respect to applications, we can refer to the work  of Liu, Yu, Wang and Li  \cite{LiYu19} and of Yi, Wang and Li \cite{YuWa19} who study the use of particular Sobolev bases for solving differential equations. The hypergeometric polynomials considered in this article are a particular type of Sobolev orthogonal polynomials linked to a higher order differential equation; for more details we refer to the work of Zagorodnyuk \cite{Za20} and the references therein. The polynomials also arise in the study of integral transforms, more particularly Riemann Liouville fractional integrals 
\cite[Chapter XIII, equation (5)]{ErMa54}. 


The particular hypergeometric polynomials we consider in this article are defined as 
$$
\lhy_n(x)= {}_2F_2(-n,1;\alpha+1,\kappa+1;x),\quad  \alpha, \kappa > -1,
$$
where
$$
 _2{F}_2 \left( \left[\begin{array}{@{}c@{}}\alpha\\ \beta \end{array}\right],
 \left[\begin{array}{@{}c@{}}\gamma\\ \delta \end{array}\right]; z\right)=
\sum_{i=0}^{\infty} \frac{(\alpha)_i (\beta)_i}{(\gamma)_i(\delta)_i} \frac{z^n}{n!}
$$
is the generalized  hypergeometric function 
and $ (\mu)_i $ is the Pochhammer symbol (shifted factorial) defined by
$(\mu)_i  =\mu(\mu+1) \cdots(\mu+i-1),$ for $ i \in \ZZ_{+}$ and $(\mu)_0=1$.

We consider nonnegative integer values of the parameter $\kappa,$ implying that the
 the  polynomials $\lhy_n(x)$ satisfy the following Sobolev orthogonality conditions ($\delta_{n,m}$ stands for the dirac delta):
$$
\int_{0}^{\infty}
\sum_{i=0}^{\kappa}\left( \frac{\kappa!}{i!}\binom{\kappa}{i} x^i\right)^2 \lhy^{(i)}_n(x)\lhy^{(i)}_m(x) e^{-x} x^{\alpha} dx =
 \frac{\kappa!^2 \Gamma(\alpha +n +1) }{  \binom{n+\alpha}{n}^2 n!}\delta_{m,n}.
$$ 

The hypergeometric polynomials $\lhy_n(x)$ are a  generalization of the Laguerre polynomials \cite{Za20} and satisfy the following $4$--term recurrence relation
\begin{equation}\label{eq:4term}
\hspace{-.2cm}\left\{
\begin{array}{@{}l}
\lhz_{-2}(x)=0,\\
\lhz_{-1}(x)=0,\\
\lhz_{0}(x)=\eta,\\
 x(e_i \lhz_{i}(x)+  f_i \lhz_{i-1}(x))= a_i \lhz_{i+1}(x)
+b_i\lhz_{i}(x)
+c_i\lhz_{i-1}(x)+d_i \lhz_{i-2}(x),
\end{array}
\right.
\end{equation}
with $i \in \NN,  $ $ \eta \in \RR \setminus\{0\}, $   and
\begin{equation}\label{eq:coef0}
\left\{\begin{array}{l}
a_i= - (i + \alpha+ 1)(i + \kappa + 1),\\
b_i=i(2i + \alpha + \kappa+ 1)+(i + \alpha +  1)(i +\kappa+1),\\
c_i=-i(3i + \alpha + \kappa),\\
d_i=(i - 1)i,\\
e_i=i + 1,\\
f_i=-i.\\
\end{array}
\right. 
\end{equation}
\begin{remark} 
Given $\alpha$ and $\kappa$, we observe that, for a given $\eta \neq 0$, the sequence $\{\lhz_{i}(x)\}_{i=0}^{\infty} $ is uniquely defined. The constant polynomial is initialized as $\eta$ and does not influence the roots nor the algorithms proposed for computing the zeros of $\lhz_{n}(x).$ The parameter $\eta$ emerges as a scaling in the recurrence relation \eqref{eq:coef0}. 
\end{remark}

The location of the  zeros of (general) hypergeometric polynomials $\lhz_{n}(x),$ and in particular the study of the asymptotic distribution is of significant interest. We refer to Zagorodnuyk \cite{Za20}, the references below, and the references therein for more information on the topic. There are various approaches that can be used to study the distribution: Srivastava, Zhou, and Wang \cite{Zh12,Sri11} study directly the recursion coefficients;  Zhou, Li, and Xu use the associated integral representation  \cite{Zh23};  one can also exploit the link with other polynomials such as the Bessel functions, which is due to Bracciali and Moreno-Balc{\'a}zar \cite{BraMo15}.
More generally applicable results are by Kuijlaars and Martínez-Finkelshtein who study the asymptotic distribution of the roots of Jacobi polynomials\cite{Ku04}; and Kuijlaars and Van Assche study the behavior of matrices in an 
asymptotic fashion \cite{n729}. The latter article \cite{n729} fits
with the research in this article in the sense that we will provide an
algorithm for efficiently computing the roots of matrices linked to
hypergeometric polynomials. To the best of our knowledge the
asymptotic eigenvalue distribution of this particular type of
polynomials has not been found yet. 
Driver and M{\"o}ller provide intervals in which the real zeros will
lie, and they provide some numerical evidence linking the zeros to
Cassini curves \cite{Dr01}.
Boggs and Duren show that the zeros, for particular ranges of
parameters, lie in a half plane and cluster on a particular loop of a
lemniscate \cite{Bo01}. Duren and Guillou \cite{Du01} refined this study  using computer graphics.



In this article we describe a weakly stable algorithm for
computing the zeros of $\lhz_{n}(x)$. The algorithm has an $\mathcal{O} (n^2)$ computational complexity and uses $\mathcal{O} (n)$ memory.
The eigenvalue problem is solved via a structure preserving and QR based
method. More information on QR algorithms can be found in Watkins'
book \cite{b333}, and structure preserving eigenvalue
algorithms are discussed in the book of Mastronardi, Van Barel, and
Vandebril \cite{b164}. To compute the roots of hypergeometric
polynomials we first transform the generalized eigenvalue problem to
a classical eigenvalue problem, which is then scaled to get a so
called Comrade matrix. Comrade matrices, which are essentially
symmetric tridiagonal plus spike matrices, are well suited for structure
preserving QR iterations. Various algorithms that execute this task
are readily available, see Eidelman, Gemignani, and Gohberg \cite{q763}, Del Corso and Vandebril \cite{VaDe09b}, and  Casulli and Robol \cite{CaRo21}. We will use this last one in this article, because it is theoretically proven to be the most reliable one.

The paper is organized as follows. In Section~\ref{sect:1} it is shown that the zeros of $\lhz_{n}(x)$ can be computed as  
 the eigenvalues of a generalized eigenvalue problem associated with the pencil $ x B_n- A_n,$ where $B_n$ and $A_n$ are both banded. Furthermore we show how to transform the pencil to a classical eigenvalue problem of a comrade matrix, that is a symmetric  tridiagonal matrix modified by a rank--one matrix.
 Section~\ref{sec:comrade} presents the basic principles of the QR based, structure preserving algorithm for computing the eigenvalues. Numerical tests are reported in Section~\ref{sect:NE}, where some rules of thumb for reliable use of the proposed method are derived.  The conclusions are to be found in Section~\ref{sect:C}.

\section{Two matrix eigenvalue problems} \label{sect:1}
In this section we show that the zeros of $\lhz_{n}(x)$ can be computed as the solution of two matrix  eigenvalue problems: a generalized and a classical eigenvalue problem.
If we write the 4-terms recurrence relation \eqref{eq:4term} for computing the polynomials $\lhz_n(x)$ in matrix form, we obtain
$$
(x B_n- A_n)\qban(x) = a_{n-1} \lhz_{n}(x)\bet_n,
$$
which we rewrite as 
$$
 \begin{array}{@{}c@{}}(xB_{n,n+1}-A_{n,n+1})\\ \hphantom{x}\end{array}\left[\begin{array}{@{}c@{}}\qban(x) \\
\hline \lhz_{n}(x)\end{array}\right]=\bf{0},
$$
where
$$
A_n=\left[\begin{array}{@{}c@{\hspace{.1cm}}c@{\hspace{.1cm}}c@{\hspace{.1cm}}c@{\hspace{.1cm}}c@{}}
b_0 & a_0 & & &  \\
c_1 &b_1 & a_1 & &  \\
d_2 &c_2 &\ddots& \ddots &   \\
					          &\ddots &\ddots &b_{n-2} & a_{n-2}  \\
	                         &              &d_{n-1}&c_{n-1} &b_{n-1}  						
\end{array}
\right], \;
B_n=\left[\begin{array}{@{}c@{\hspace{.1cm}}c@{\hspace{.1cm}}c@{\hspace{.1cm}}c@{\hspace{.1cm}}c@{}}
 e_0 & & \\
 f_1 & e_1      &         & \\
     & \ddots & \ddots  & \\
     &        & f_{n-2} &e_{n-2} \\
		 &        &         & f_{n-1} &e_{n-1}
\end{array}
\right], 
$$
and
$$
\qban(x)=
\left[\begin{array}{@{}c@{}}
\lhz_0(x) \\
\lhz_1(x) \\
\lhz_2(x) \\
\vdots \\
\lhz_{n-2}(x) \\
\lhz_{n-1}(x) \\
\end{array}
\right], \;
B_{n,n+1}=[B_n\mid {\bf 0}], \; \mbox{\rm and}\; A_{n,n+1}=[ A_n\mid a_{n-1}\bet_n].$$ 
\\
Therefore, $ x^\star$ is a zero of  $\lhz_{n}(x)$ if and only if $ x^\star$ is a generalized eigenvalue of the pencil
\begin{equation}\label{eq:geneig0}
x B_n - A_n.
\end{equation}
The following theorem states that the zeros of $\lhz_{n}(x)$ are also the eigenvalues of a symmetric tridiagonal matrix plus a rank--one matrix. In practice it is often advised, for reasons of numerical reliability, to solve the generalized eigenvalue problem \eqref{eq:geneig0} directly. In this particular setting, however, we will see that 
the problem is heavily structured and allows us to transform it to a structured classical eigenvalue problem.

\begin{theorem}The matrix $ B_n$ is nonsingular and
\begin{equation}\label{eq:semi}
B_n^{-1}=
D_n^{-1} M_n,\;
D_n=\left[ \begin{array}{ccccc}
1       &        &        &       &       \\
        & 2      &        &       &       \\
        &        & \ddots &       &       \\
				&        &        & n-1   &       \\
				&        &        &       & n
				\end{array} \right],
				\;
M_n=\left[ \begin{array}{ccccc}
1       &        &        &       &       \\
1       & 1      &        &       &       \\
\vdots  & \cdots & \ddots &       &       \\
1		  	& 1      & \cdots & 1     &       \\
1 			& 1      & \cdots & 1     &   1
				\end{array} \right].
				\quad				
\end{equation}
Moreover,
\begin{equation}\label{eq:comrade}
{X}_n=B_n^{-1}A_n =
\left[\begin{array}{@{}c@{\hspace{.1cm}}c@{\hspace{.1cm}}c@{\hspace{.1cm}}c@{\hspace{.1cm}}c@{}}
\hat v_0 & \hat a_0 & & &  \\
\hat v_1 &\hat b_1 &\hat  a_1 & &  \\
\hat v_2 &\hat c_2 &\ddots& \ddots &   \\
\vdots         &       &\ddots &\hat b_{n-2} &\hat  a_{n-2}  \\
\hat v_{n-1}	                         &              &       &\hat c_{n-1} &\hat b_{n-1} 
\end{array}
\right]
\end{equation}
with
$$
\begin{array}{lll}
\hat a_{i-2} &= -\frac{(i-1+\alpha)(i-1+\kappa)}{i-1},& i =2,3,\ldots, n\\
\hat b_{i-1} &=2i-1+\alpha+\kappa,& i =2,3,\ldots, n,\\
\hat c_{i} &= -i,& i =2,3,\ldots, n-1,\\
\hat v_0&= (\alpha+1)(\kappa+1),\\
\hat v_1&=\frac{\alpha \kappa}{2}-1,\\
\hat v_i&=\frac{\alpha \kappa}{i+1}, & i =2,3,\ldots, n-1.\\
\end{array}
$$
\end{theorem}

\begin{proof}
Consider a nonsingular bidiagonal matrix, with all $\tau$'s and $\upsilon$'s different from zero:
$$
B=\left[\begin{array}{@{}c@{\hspace{.1cm}}c@{\hspace{.1cm}}c@{\hspace{.1cm}}c@{\hspace{.1cm}}c@{}}
 \tau_1 & & \\
 \upsilon_1 & \tau_2      &         & \\
     & \ddots & \ddots  & \\
     &        & \upsilon_{n-2} &\tau_{n-1} \\
		 &        &         & \upsilon_{n-1} &\tau_{n}
\end{array}
\right].
$$
The inverse of this matrix is a so-called semiseparable matrix \cite[Th.~4.5]{b163}. Because the $\upsilon$'s are nonzero, the inverse is the lower triangular part of a rank one matrix. 
We get that
$$
B^{-1} = M^{-1} N,
$$
where $M$ and $N$ are as follows
$$
M=
\left[\begin{array}{@{}c@{\hspace{.1cm}}c@{\hspace{.1cm}}c@{\hspace{.1cm}}c@{\hspace{.1cm}}c@{}}
 \tau_1 & & \\
  & \tau_2      &         & \\
     &        & \ddots  & \\
     &        &  &\tau_{n-1} \\
		 &        &         & &\tau_{n}
\end{array}
\right],
$$
$$
N=
\left[\begin{array}{@{}cccccc@{}} 
1             &   &       &         & \\
(-\frac{\upsilon_1}{\tau_1})                             &1   &  &  & \\
(-\frac{\upsilon_1}{\tau_1})(-\frac{\upsilon_2}{\tau_2}) &(-\frac{\upsilon_2}{\tau_2}) & 1    \\ 
(-\frac{\upsilon_1}{\tau_1})(-\frac{\upsilon_2}{\tau_2})(-\frac{\upsilon_3}{\tau_3}) &(-\frac{\upsilon_2}{\tau_2})(-\frac{\upsilon_3}{\tau_3}) &(-\frac{\upsilon_3}{\tau_3}) & 1  \\
\vdots	   &    &    & \ddots  &\ddots \\
(-\frac{\upsilon_1}{\tau_1})\cdots(-\frac{\upsilon_{n-1}}{\tau_{n-1}}) &(-\frac{\upsilon_2}{\tau_2})\cdots(-\frac{\upsilon_{n-1}}{\tau_{n-1}})&(-\frac{\upsilon_3}{\tau_3})\cdots(-\frac{\upsilon_{n-1}}{\tau_{n-1}}) & \hphantom{..}\cdots\hphantom{..} & (-\frac{\upsilon_{n-1}}{\tau_{n-1}}) & 1  \\ 
\end{array}
\right].
$$
Hence, Equation \eqref{eq:semi} follows taking the correct diagonal and subdiagonal values of $B$ (see \eqref{eq:coef0}) into account. 

For part two of the proof we consider first the matrix $A_n$.
The $ i$th column of $ A_n $,  $ i=2,\ldots, n,$ is given by
$$
\left[\begin{array}{c}
\mathbf{0}_{i-2}\\
a_{i-2} \\
b_{i-1}\\
c_{i}\\
d_{i+1}\\
\mathbf{0}_{n-i-2}
\end{array}
\right]
=
\left[\begin{array}{c}
\mathbf{0}_{i-2}\\
-(i-1+\alpha)(i-1+\kappa)\\ 
(i-1)(2i-1+\alpha+\kappa)+(i+\alpha)(i+\kappa)\\
-i(3i+\alpha+\kappa)\\
i(i+1)\\
\mathbf{0}_{n-i-2}
\end{array}
\right].
$$ 
Hence,
\begin{equation}\label{eq:cond1}
\begin{array}{llll}
\hat a_{i-2} &= B_n^{-1}(i-1,:) A_n(:,i)&=\frac{a_{i-2}}{i-1}&=-\frac{(i-1+\alpha)(i-1+\kappa)}{i-1},\\
\hat b_{i-1} &= B_n^{-1}(i,:) A_n(:,i)&=\frac{a_{i-2}+b_{i-1}}{i}&=2i-1+\alpha+\kappa,\\
\hat c_{i} &= B_n^{-1}(i+1,:) A_n(:,i)&=\frac{a_{i-2}+b_{i-1}+c_{i}}{i+1}&=-i,\\
\hat v_0&=B_n^{-1}(1,:) A_n(:,1)&=b_0 & = (\alpha+1)(\kappa+1),\\
\hat v_1&=B_n^{-1}(2,:) A_n(:,1)& \frac{b_0+c_1}{2}&=\frac{\alpha \kappa}{2}-1,\\
\hat v_i&=B_n^{-1}(i,:) A_n(:,1)& \frac{b_0+c_1+d_2}{i}&=\frac{\alpha \kappa}{i+1}, \;\; i =2,3,\ldots, n-1.\\
\end{array}
\end{equation}
Moreover,
$$
 X(i,j) = 0\;\; \mbox{\rm if} \;\;  i=1,\ldots, n-2, \;\;  i+2 <j\le n 
$$
since $D_n$ and $ M_n $ are lower triangular and $A_n $ is lower Hessenberg.

Finally, for $  j=2, \ldots, n, \;\; j+2 \le i \le 4, $
\begin{eqnarray*}
 X(i,j) & =& B_n^{-1}(i,:) A_n(:,j)
= \frac{1}{i} 
[\underbrace{\begin{array}{ccc}1, & \ldots, &  1\end{array}}_{i},\underbrace{\begin{array}{ccc}0, & \ldots, &  0\end{array}}_{n-i}]
\left[\begin{array}{c}
\mathbf{0}_{j-2}\\
a_{j-2} \\
b_{j-1}\\
c_{j}\\
d_{j+1}\\
\mathbf{0}_{n-j-2}
\end{array}
\right]\\
&=&\frac{a_{j-2}+b_{j-1}+c_{j}+d_{j+1}}{i} = 0 
\end{eqnarray*}
This concludes the proof by stating that $A_n$ satisfies \eqref{eq:comrade}
\end{proof}
The matrix $X_n$ is nothing else than a tridiagonal matrix plus a spike in the first column.
Since $ \hat{a}_i < 0,\; i=0,\ldots,n-2, $ and $ \hat{c}_i < 0,\; i = 2,\ldots, n-1, $ it is always possible to construct a diagonal matrix 
 $ {D}_n=\mbox{\rm diag}(\delta_1,\ldots,\delta_n) $ of order $n $ such that  $ C_n={D}_n {X}_n  {D}_n^{-1}$  has its main tridiagonal part symmetric, which means that $C_n$ without the spike is symmetric, or in other words the super- and subdiagonal are equal. This matrix is called a comrade matrix.
The scaling allows us thus to write $C_n$ as the sum of a symmetric tridiagonal plus rank one matrix $T_n + \mathbf{w}_1\mathbf{e}_1^T$ with $T_n$ symmetric and $\mathbf{w}_1\mathbf{e}_1^T$ the spike: 
\begin{equation}\label{eq:rankone}
C_n={D}_n {X}_n  {D}_n^{-1}=
\left[\begin{array}{@{}c@{\hspace{.05cm}}|c@{\hspace{.1cm}}c@{\hspace{.1cm}}c@{\hspace{.1cm}}c@{}}
\tilde v_0 & \tilde a_0 & & &  \\
\hline
\\[-.35cm] 
\tilde a_0 &\hat b_1 &\tilde  a_1 & &  \\
\tilde v_2 &\tilde a_1 &\ddots& \ddots &   \\
\vdots         &       &\ddots &\hat b_{n-2} &\tilde  a_{n-2}  \\
\tilde v_{n-1}	                         &              &       &\tilde a_{n-2} &\hat b_{n-1}  						
\end{array}
\right] =
T_n + \mathbf{w}_1\mathbf{e}_1^T,
\end{equation}
%
with
$$
T_n =
\left[\begin{array}{@{}c@{\hspace{.1cm}}c@{\hspace{.1cm}}c@{\hspace{.1cm}}c@{\hspace{.1cm}}c@{}}
\tilde v_0 & \tilde a_0 & & &  \\
\tilde a_0 &\hat b_1 &\tilde  a_1 & &  \\
         &\tilde a_1 &\ddots& \ddots &   \\
         &       &\ddots &\hat b_{n-2} &\tilde  a_{n-2}  \\
	                         &              &       &\tilde a_{n-1} &\hat b_{n-1}	
															\end{array}
\right],\;\; 
 \mathbf{w}_1=
\left[
\begin{array}{c}
0\\
\delta_2 \tilde v_1-\tilde a_0\\
\delta_3 \tilde v_2\\
\vdots \\
\delta_n\tilde v_{n-1}\\
\end{array}
\right],\;\;
\mathbf{e}_1=
\left[
\begin{array}{c}
1\\
0\\
\vdots \\
0\\
0
\end{array}
\right].
$$
The coefficients  $\tilde  a_i,\; i=0, \ldots,{n-2}, $ and $\delta_i, i=1,\ldots, n $ can be computed recursively as follows.
Initialize $\delta_1=1$ and $\delta_2=1$, then we can use the following iteration to compute the remaining values for $i=3,\ldots,n$:
%
                \begin{equation}
                \label{eq:cond2}
                \delta_{i}=\delta_{i-1} \sqrt{ \hat a_{i-2}/\hat c_{i-1}}\quad \mbox{and} \quad
                \tilde a_{i-2}=\sqrt{\hat a_{i-2} \hat c_{i-1}}.
                \end{equation}
Taking the values of  $\hat a_{i} $ and $\hat c_{i}$ from \eqref{eq:cond1}, we can rewrite 
\eqref{eq:cond2} as follows, for $i=2,\ldots,n$:
\begin{equation*}
                \delta_{i+1}=\delta_{i}\sqrt{ \frac{i+\alpha}{i}\frac{i+\kappa}{i}}
                \quad\mbox{and}\quad
                \tilde a_{i-1}= \sqrt{ (i+\alpha)(i+\kappa)}.
\end{equation*}

We want to remark that care has to be taken when balancing a matrix
for eigenvalue computations \cite{rv004,q599}. In the numerical
experiments we illustrate that balancing increases the sensitivity of
the problem. 

\section{Computation of the zeros of $\lhz_{n}(x)$}
\label{sec:comrade}
As shown, in Section~\ref{sect:1}, 
the zeros of $\lhz_{n}(x)$  are the eigenvalues of a symmetric tridiagonal matrix $T_n $ plus a rank--one modification $\mathbf{w}_1\mathbf{e}_1^T.$
The eigenvalues of such matrices can be computed in a  fast and backward stable fashion, i.e., with  $\mathcal{O}(n^2)$  computational complexity, where $ n $ is the size of the matrix, by variants of the QR algorithm exploiting the  structure of the involved matrices.
 In particular, to compute the eigenvalues of (\ref{eq:rankone}), we consider the
  algorithm by Casulli and Robol \cite{CaRo21},    a structure preserving variant of the QR algorithm.
	
	The classical QR algorithm was  originally developed by Francis and Kublanovskaya \cite{Fr61,Ku62}, and is  described in detail   
	in Watkins \cite{b333}. Essentially, the QR algorithm computes the Schur decomposition $X=Q U Q^T$, of a given input matrix $X \in \RR^{n \times n}$, where $U$ is upper triangular and $Q$ orthogonal\footnote{For simplicity, we have restricted ourselves to real matrices and orthogonal similarity transformations. All of this can be generalized to the complex setting without loss of generality.}. The eigenvalues are revealed on the main diagonal of the upper triangular matrix $U$. Of course, the algorithm is iterative in nature; it takes thus several iterations before the factorization is --numerically-- achieved. 
	Essentially, one starts with $X^{(0)}:= X$, and at each iteration, called  QR-step, applies an orthogonal similarity transformation  
	to get $X^{(i)}=(Q^{(i-1)})^T X^{(i-1)} Q^{(i-1)}$, $ i=1, 2, \ldots$. 

 Every  QR-step requiring $ \mathcal{O}(n^3)$ floating point operations,   bringing  the matrix $X^{(i)}$ closer and closer to upper triangular form, until we are satisfied and have achieved a certain precision. To lower the computational cost, the matrix $ X $ is first transformed into a similar upper Hessenberg matrix, 
i.e., a matrix with all the entries  equal to zero below the first subdiagonal.

 The QR algorithm then starts iterating on this Hessenberg matrix, preserving the Hessenberg structure throughout all the iterations, and making the elements of the first subdiagonal closer and closer to zero. If applied to a Hessenberg matrix,  each QR--step requires  $\mathcal{O}(n^2)$ floating point operations, 
	and  the QR  algorithm exibits  $\mathcal{O}(n^3)$ computational complexity. 
	
	The QR algorithm  can be similarly implemented considering  lower Hessenberg matrices instead, i.e., matrices with all the entries  equal to zero above the first superdiagonal. In this case, the sequence of matrices generated by the QR algorithm converges to a lower triangular matrix\footnote{This can be achieved easily by working on the transpose as the transposing a matrix does not change its eigenvalues.}.
	
	In case the involved matrix is  (\ref{eq:rankone}), the
  variant of the QR algorithm developed by Casulli and Robol \cite{CaRo21}   exploits the structure of the matrix, requiring  $\mathcal{O}(n)$ floating point operations for each QR-step, and  $\mathcal{O}(n^2)$ overall computational complexity.
	
	Let us shortly describe one QR-step of the latter algorithm. 
The involved matrix  (\ref{eq:rankone})
 has  the form $C_n=T_n+\mathbf{w}_1 \mathbf{e}_1^T$, i.e.,   a symmetric  tridiagonal $T_n$ plus the rank--one matrix $\mathbf{w}_1 \mathbf{e}_1^T$. Therefore,   $C_n$  is already in lower  Hessenberg form.

Let $ C_n^{(0)} = C_n$. Applying one QR-step to $  C_n^{(0)}$, 
we get $$C_n^{(1)}={Q^{(0)}}^T  C_n^{(0)}  Q^{(0)}={Q^{(0)}}^T T_n Q^{(0)}  + ({Q^{(0)}}^T \mathbf{w}_1) \, ({Q^{(0)}}^T \mathbf{e}_1)^T.$$ Clearly, ${Q^{(0)}}^T T_n Q^{(0)}$ is again symmetric, and  $({Q^{(0)}}^T \mathbf{w}_1) \, ({Q^{(0)}}^T \mathbf{e}_1)^T$ again of rank one form. Also the matrix $C_n^{(1)}$ will again, by construction, be in lower Hessenberg form. Moreover, the orthogonal matrix $Q^{(0)}$ is the product of $ n-1$  Givens rotations. Hence, the computation of  $C_n^{(1)} $ requires   $\mathcal{O}(n)$ floating point operations.

  Furthermore, these three properties, which are, symmetry, rank--one, and Hessenberg structure, will be maintained throughout all the iterations. This allows to develop a cheap, $\mathcal{O}(n)$ storage scheme for storing the matrices. As a consequence, a structure preserving QR algorithm will only take  $\mathcal{O}(n^2)$ floating point operations, instead   $\mathcal{O}(n^3),$  for computing all the eigenvalues of $ C_n.$

Figure~\ref{timings} shows the time in seconds for running the
classical eigenvalue solver and the fast one. Starting from size 1000
the fast algorithm clearly wins. The experiments were done in Matlab,
where the fast algorithm is encoded via Mex-files. The figure is
merely to illustrate that the fast algorithm becomes significantly
faster than the classical one at a certain point; we do not claim
anything on the crossover point, this is highly depending on the
programming language/implementations used.
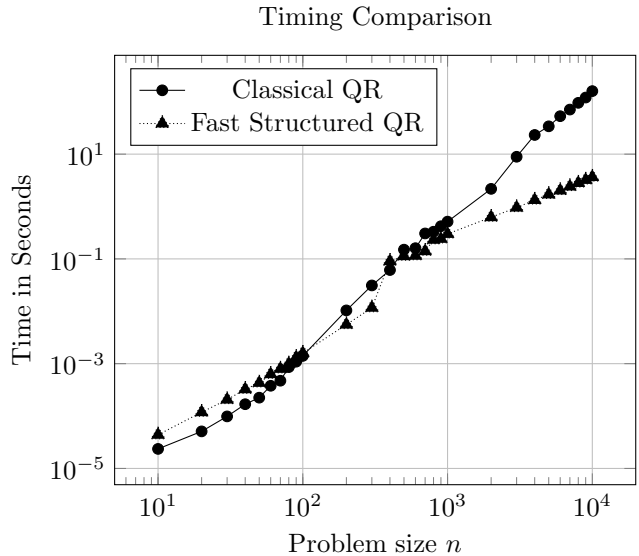
\begin{figure}
\centering
\caption{Comparison of timings in seconds between the classical order
  $n^3$ QR algorithm and the fast $n^2$ structured one.}
\label{timings}
\begin{tikzpicture}
  \begin{loglogaxis}[
    title={Timing Comparison},
        xlabel={Problem size $n$},
        ylabel={Time in Seconds},
        grid=major,
        legend pos=north west,
        every axis plot/.append style={black, thin, mark size=1pt}, 
    ]

    \addplot[black,mark=*,mark size=2pt] table [x index=0, y index=1, col sep=space] {Experiment13/data_exp_timings.txt};
    \addplot[black, mark=triangle*,mark size=3pt,densely dotted] table [x index=0, y index=2, col sep=space] {Experiment13/data_exp_timings.txt};
    \addlegendentry{Classical QR}
    \addlegendentry{Fast Structured QR}

    \end{loglogaxis}
\end{tikzpicture}
\end{figure}


\section{Numerical experiments}\label{sect:NE}
In this section we will compare manners to compute the eigenvalues and compare them with respect to their numerical
stability. We will use the fastest method of the three to examine some properties of the eigenvalues, i.e., the zeros of the hypergeometric polynomials
$${}_2F_2(-n,1;\alpha+1,\kappa+1;x),\quad  \alpha, \kappa \geq -1.$$

\noindent The three algorithms are the following ones:
\begin{itemize}
\item The first algorithm takes the unsymmetrized matrix $X$ of size $n$ for a given value of $\alpha$ and $\kappa$ and solves the corresponding eigenvalue problem using the classical, not structure exploiting QR algorithm. 
For our purpose we have used the Matlab solver \texttt{eig}.

\item The second algorithm uses the symmetrized Comrade matrix $C$ as input and solves the
corresponding eigenvalue problem using the non structured QR algorithm. Again we use the Matlab solver \texttt{eig}.

\item The third algorithm uses the symmetrized matrix $C$ as input and
  relies on a fast structure preserving QR algorithm. We opted to take
  the fast algorithm of Casulli and Robol \cite{CaRo21}.

\end{itemize}
We remark that we did not turn off the balancing option of the Matlab solver \texttt{eig}, which means that a particular balancing strategy is applied in order to enhance the accuracy. The fast algorithm does not allow for additional balancing. We will refer to the three approaches as Algorithms 1,2, and 3 respectively.


\subsection{Stability analysis}

To study the stability of these three algorithms we work in double precision and in multiple precision to construct the matrices and compute the eigenvalues.
Figure~\ref{fig:stab:1} shows the maximum absolute error of the eigenvalues for size $n = 100$ and $\alpha$ and $\kappa$ values taken from $-1$ to $5$ by steps of $0.1$. We assume the eigenvalues computed in multiple precision to be correct and determine the error based on those. In this first test we only considered the non-structured eigenvalue solver on the unsymmetrized problem, thus Algorithm 1.
One can see that this error increases most at the diagonal of the $\alpha$-$\kappa$ plane. Hence we take $\alpha=\kappa$ in the next figures.

\begin{center}
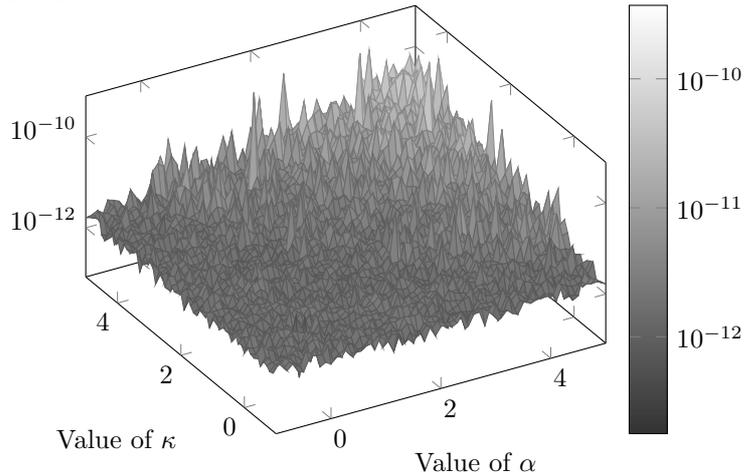
\begin{figure}[hbtp]
    \caption{Maximum absolute error on the eigenvalues for the non-structured QR algorithm on the nonsymmetrized matrix $X$. }
    \label{fig:stab:1}
    \centering
\begin{tikzpicture}
  \begin{axis}[
     view={-30}{45},
     colormap={bw}{gray(0cm)=(0.2);gray(1cm)=(1)},
     xlabel={Value of $\alpha$},
     ylabel={Value of $\kappa$},
     zticklabels={$10^{-14}$, $10^{-12}$, $10^{-10}$},
     colorbar,
     colorbar style={
            ytick={-12, -11, -10}, 
            yticklabels={$10^{-12}$, $10^{-11}$,$10^{-10}$} 
        },
     mesh/ordering=x varies
    ]
    \addplot3[surf] table [col sep=space,row sep=newline] {./Experiment10/data_exp10_1.txt};
 \end{axis}
\end{tikzpicture}
\end{figure}
\end{center}

In Figure~\ref{fig:stab:2}, the four subplots show this error for $\alpha = \kappa$ and for sizes $n= 100, 300, 400$ and $1000$ for the three algorithms. Note that the legend is the same for all figures, but the y-axis scale, depicting the error, varies.
We compare this absolute error with the product of the norm of the matrix $X$, the machine precision $\epsilon$ and the maximum of the condition number of the eigenvalues for $X$, and we do the same for $C$.
The two latter measures are a prediction on the upper bound of the error to be expected.
The condition number of an eigenvalue is $\frac{1}{w^T v}$ where $v$ and $w$
are the right and left eigenvectors of unit norm corresponding to the eigenvalue.

\begin{figure}[h!tb]
\centering
\caption{Comparison of the absolute error of the three algorithms and two additional curves representing upper bounds for both nonsymmetric and symmetric solvers. Note that the legend holds for all subfigures, but the y-axis scale differs significantly.}
\label{fig:stab:2}
  \begin{tikzpicture}[scale=0.7]
     \begin{semilogyaxis}[
    title={Errors for size $n=100$},
        xlabel={$\alpha=\kappa$},
        ylabel={Error},
        grid=major,
        legend pos=north west,
       every axis plot/.append style={mark size=1.7pt}, 
       width=9cm,
       height=9cm
    ]

\addplot[black,densely dotted, mark=*] table [x index=0, y index=1, col sep=space] {Experiment11/data_exp16_1.txt};
\addplot[black, densely dashed, mark=diamond*] table [x index=0, y index=2, col sep=space] {Experiment11/data_exp16_1.txt};
\addplot[black,solid, mark=square*] table [x index=0, y index=3, col sep=space] {Experiment11/data_exp16_1.txt};
\addplot[black, dash dot, mark=triangle*] table [x index=0, y index=4, col sep=space] {Experiment11/data_exp16_1.txt};
\addplot[black, solid,mark=x] table [x index=0, y index=5, col sep=space] {Experiment11/data_exp16_1.txt};

\addlegendentry{Fast Alg.\  on$C$}
\addlegendentry{\texttt{eig}($C$)}
\addlegendentry{Bound Sym.\ $C$}
\addlegendentry{\texttt{eig}($X$)}
\addlegendentry{Bound Unsym.\ $X$}

\end{semilogyaxis}
\end{tikzpicture}
  \begin{tikzpicture}[scale=0.7]
     \begin{semilogyaxis}[
    title={Errors for size $n=300$},
        xlabel={$\alpha=\kappa$},
        ylabel={Error},
        grid=major,
        legend pos=north west,
       every axis plot/.append style={mark size=1.7pt}, 
       width=9cm,
       height=9cm
    ]

\addplot[black, densely dotted, mark=*] table [x index=0, y index=1, col sep=space] {Experiment11/data_exp16_3.txt};
\addplot[black, densely dashed, mark=diamond*] table [x index=0, y index=2, col sep=space] {Experiment11/data_exp16_3.txt};
\addplot[black, solid, mark=square*] table [x index=0, y index=3, col sep=space] {Experiment11/data_exp16_3.txt};
\addplot[black, dash dot, mark=triangle*] table [x index=0, y index=4, col sep=space] {Experiment11/data_exp16_3.txt};
\addplot[black, solid,mark=x] table [x index=0, y index=5, col sep=space] {Experiment11/data_exp16_3.txt};


\end{semilogyaxis}
\end{tikzpicture}

  \begin{tikzpicture}[scale=.7]
     \begin{semilogyaxis}[
    title={Errors for size $n=400$},
        xlabel={$\alpha=\kappa$},
        ylabel={Error},
        grid=major,
        legend pos=north west,
       every axis plot/.append style={mark size=1.7pt}, 
       width=9cm,
       height=9cm
    ]

\addplot[black, densely dotted, mark=*] table [x index=0, y index=1, col sep=space] {Experiment11/data_exp16_4.txt};
\addplot[black, densely dashed, mark=diamond*] table [x index=0, y index=2, col sep=space] {Experiment11/data_exp16_4.txt};
\addplot[black, solid, mark=square*] table [x index=0, y index=3, col sep=space] {Experiment11/data_exp16_4.txt};
\addplot[black, dash dot, mark=triangle*] table [x index=0, y index=4, col sep=space] {Experiment11/data_exp16_4.txt};
\addplot[black, solid,mark=x] table [x index=0, y index=5, col sep=space] {Experiment11/data_exp16_4.txt};


\end{semilogyaxis}
\end{tikzpicture}
  \begin{tikzpicture}[scale=.7]
     \begin{semilogyaxis}[
    title={Errors for size $n=1000$},
        xlabel={$\alpha=\kappa$},
        ylabel={Error},
        grid=major,
        legend pos=north west,
       every axis plot/.append style={mark size=1.7pt}, 
       width=9cm,
       height=9cm
    ]

\addplot[black, densely dotted, mark=*] table [x index=0, y index=1, col sep=space] {Experiment11/data_exp16_10.txt};
\addplot[black, densely dashed, mark=diamond*] table [x index=0, y index=2, col sep=space] {Experiment11/data_exp16_10.txt};
\addplot[black, solid, mark=square*] table [x index=0, y index=3, col sep=space] {Experiment11/data_exp16_10.txt};
\addplot[black, dash dot, mark=triangle*] table [x index=0, y index=4, col sep=space] {Experiment11/data_exp16_10.txt};
\addplot[black, solid,mark=x] table [x index=0, y index=5, col sep=space] {Experiment11/data_exp16_10.txt};


\end{semilogyaxis}
\end{tikzpicture}
\end{figure}
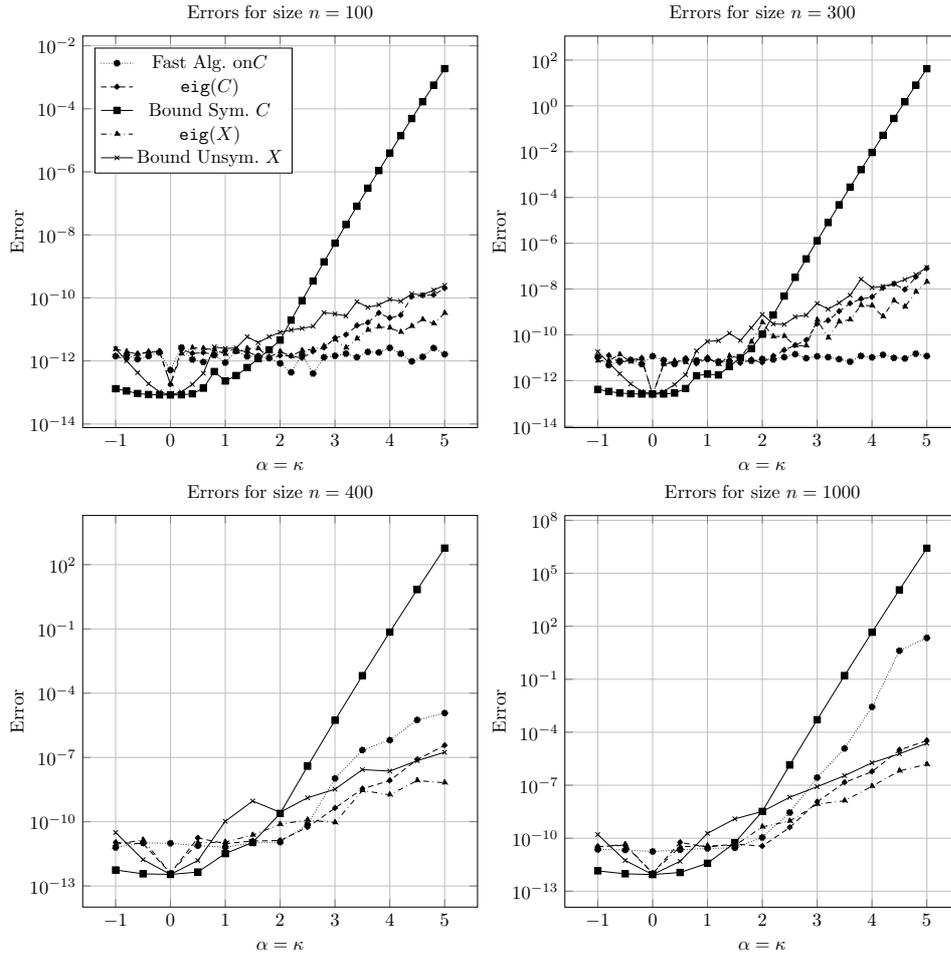

Algorithm 1 and 2 that use \texttt{eig} of Matlab use balancing by default.
Their behavior is similar to the error measure for the unsymmetrized matrix because this is
already a well-balanced matrix with respect to the eigenvalue problem.
Because the fast algorithm is not using balancing, one could expect a behavior following 
the error measure for the symmetrized matrix but it behaves much better when $n < 400$.
For higher values of $n$ it starts to follow the error measure for the symmetrized matrix,
which is very ill-balanced with respect to the eigenvalue problem.
One can conclude that the three algorithms are weakly stable for their
corresponding problem.

Summarizing we get the following result.
For $\alpha = \kappa$ between $-1$ and $2.5$ the three methods are comparable
with respect to accuracy but if $n > 1000$ algorithm 3 becomes faster than the other two when $n$ grows.
For $\alpha = \kappa$ between $2.5$ and $5$ and $n < 400$ the fast algorithm  is preferable
because it gives the highest accuracy.
However, if $n >= 400$ the other two algorithms are the best where \texttt{eig}
applied to $X$ is slightly better compared to $C$.

\subsection{Range of stability for the fast algorithm}

For the fast algorithm, one can expect visible changes in the plot of the spectrum
for values of $\alpha = \kappa$ larger than $4$, due to the
difficulties of solving the problem
For $\alpha = \kappa = 4.4$ and $n = 1000$ the eigenvalues computed in double and in
multiple precision are plotted Figure~\ref{spectrum}.
Note that for double precision, due to the ill-conditioning of the problem, 4 
eigenvalues are moved away from the circular region of the eigenvalues.

\begin{figure}[htb]
\centering
\caption{High precision and double precision plot of part of the
  spectrum, to illustrate the ill-conditioning. The
  eigenvalues computed in double precision deviate from the correct ones.}
\label{spectrum}
\begin{tikzpicture}[scale=.8]
    \begin{axis}[scale=.8,
        grid=major, 
        legend pos=north east, 
        xmin=-10,xmax=100,
        ymin=-5,ymax=5,
        width=12cm,
        height=12cm,
        title={Plot of the spectrum for $\alpha=\kappa=4.4$}
    ]

    \addplot[
        only marks, 
        mark=*, 
        mark size=1pt, 
        color=black
    ] table [x index=0, y index=1, col sep=space] {Experiment9/data_exp92.txt};
    \addlegendentry{Multiple Precision}

    \addplot[
        only marks, 
        mark=o, 
        mark size=2.5pt, 
        thick, 
        color=black
    ] table [x index=0, y index=1, col sep=space] {Experiment9/data_exp91.txt};
    \addlegendentry{Double Precision}

    \end{axis}
  \end{tikzpicture}
  \end{figure}

We can conclude that the fast algorithm $3$ can be used to study the behaviour of the eigenvalues depending on $n$, $\alpha$ and $\kappa$ for values of $\alpha$ and $\kappa$ smaller than $4$ or $5$.

\subsection{Separation between complex and real spectrum}
In Figure~\ref{separation} the division lines between the regions in the 
$\alpha-\kappa$-plane are indicated where all the eigenvalues are real and where
there are some of the eigenvalues that are non-real. We show this division line
when the size of the matrix is $10$, $100$ or $1000$.
For $\alpha$ or $\kappa$ equal to $-1$ or $0$ all eigenvalues are real.
In the bottom left region all eigenvalues are real while in the other 3 regions
there are non-real eigenvalues. We remark that the curve for $n=10000$
could only be computed by using the fast algorithm.

\begin{figure}[htb]
\centering
\caption{Separation lines in the $\alpha$ $\kappa$-plane to
  distinguish two regions. Below left the eigenvalues are real,
  crossing the line results in complex eigenvalues.}
\label{separation}
\begin{tikzpicture}[scale=.8]
  \begin{axis}[
    scale=.9,
        grid=major,
        xlabel={Value of $\alpha$},
        ylabel={Value of $\kappa$},
        title={Separation lines between real and complex roots},
        restrict y to domain=-1:100,
        restrict x to domain=-1:100,
        ymax=5,ymin=-1,
        xmax=5,xmin=-1,
        axis equal,
        width=10cm,
        height=10cm,
        xtick={-1,0,1,2,3,4,5},
        ytick={-1,0,1,2,3,4,5},
    ]
  
    \addplot[mark=none,dotted,very thick] table [col sep=tab, x index=0, y index=1] {./Experiment12/curve2.txt};
    \addplot[mark=none,dashed,very thick] table [col sep=tab, x index=0, y index=1] {./Experiment12/curve3.txt};
    \addplot[mark=none,very thick,dashdotted] table [col sep=tab, x index=0, y index=1] {./Experiment12/curve1.txt};
    \addplot[mark=none,very thick] table [col sep=tab, x index=0, y index=1] {./Experiment12/curve4.txt};
    
    \addlegendentry{$n=10$}
    \addlegendentry{$n=100$}
    \addlegendentry{$n=1000$}
    \addlegendentry{$n=10.000$}
    \end{axis}
  \end{tikzpicture}
\end{figure}

\subsection{Growth of the spectrum}

Figure~\ref{max:eig} shows the maximum of the real part of the eigenvalues which is also the maximum real eigenvalue for sizes 
$n = 10:10:100; 200:100:1000; 2000:1000:10000$.
The different lines correspond to different values for $\alpha$ and $\kappa$.
The parameter $\alpha$ is taken from $-1:1:5$ as well as the parameter $\kappa$.
For small values of the size of the matrix $n$ there is some variation
depending on the values of $\alpha$ and $\kappa$ while this becomes smaller for
increasing values of $n$.

\begin{figure}[htb]
\centering
\caption{Maximum real part of the eigenvalues.}
\label{max:eig}

\begin{tikzpicture}[scale=.9]
  \begin{loglogaxis}[scale=1,
    title={The maximum real eigenvalue depends almost linearly on $n$},
        xlabel={Problem size $n$},
        ylabel={Maximum Real Eigenvalue},
        grid=major,
        legend pos=north west,
        every axis plot/.append style={black, thin, mark=* , mark size=1pt}, 
    ]

    \foreach \i in {1,...,50} {
        \addplot[black] table [x index=0, y index=\i, col sep=space] {Experiment2/data_exp2.txt};
    }

    \end{loglogaxis}
\end{tikzpicture}
\end{figure}


Fig~\ref{fig:max:im} we show the maximum of the absolute value of the imaginary part of the eigenvalues for the size of the matrix $n = 1000$ and for $\alpha$ and $\kappa$ values taken from $-1:0.1:+4$.

\begin{figure}[htb]
\centering
\caption{Maximum absolute value of imaginary part of the eigenvalues.}
\label{fig:max:im}

\begin{center}

\begin{tikzpicture}[scale=.9]
  \begin{axis}[scale=1,
     view={-30}{45},
     colormap={bw}{gray(0cm)=(0.2);gray(1cm)=(1)},
     xlabel={Value of $\alpha$},
     ylabel={Value of $\kappa$},
     title={Maximum absolute value of the imaginary spectrum part},
     colorbar,
     mesh/ordering=x varies
    ]
    \addplot3[surf] table [col sep=space,row sep=newline] {./Experiment3/data_exp31.txt};
 \end{axis}
\end{tikzpicture}

\end{center}
\end{figure}
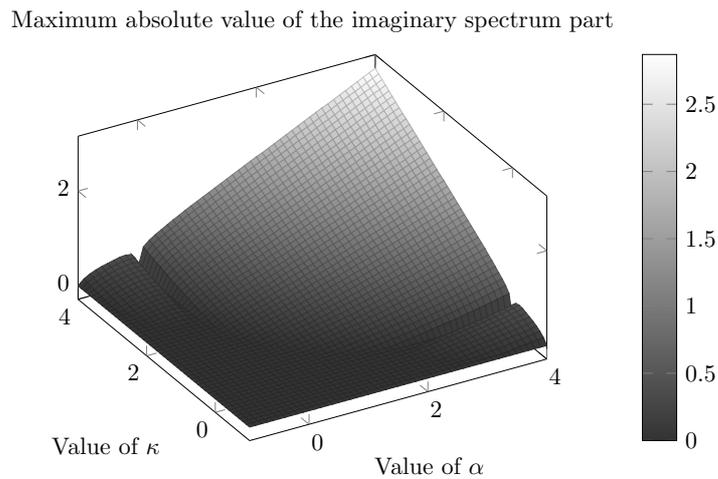

\pagebreak
Figure~\ref{min:eig} shows the minimum of the real part of all the eigenvalues where
the size of the matrix $n$ is taken from $10:10:100; 200:100:1000; 2000:1000:10000$
while $\alpha$ and $\kappa$ are taken from $0:1:3$ where we have excluded the value
$\alpha$ or $\kappa$ equal to $-1$. In this case the minimum of the real part of the
eigenvalues is $0$ corresponding to the eigenvalue $0$.

\begin{figure}[htb]
\centering
\caption{The minimum real part of the spectrum.}
\label{min:eig}
\begin{center}
  \begin{tikzpicture}[scale=.9]
    \begin{loglogaxis}[
      scale=1,
    title={The minimum real eigenvalue depends almost linearly on $n$},
        xlabel={Problem size $n$},
        ylabel={Minimum value},
        grid=major,
        legend pos=north west,
        every axis plot/.append style={black, thin, mark=* , mark size=1.5pt}, 
    ]

    \foreach \i in {1,...,16} {
        \addplot[black] table [x index=0, y index=\i, col sep=space] {Experiment4/data_exp4.txt};
    }

    \end{loglogaxis}
\end{tikzpicture}
\end{center}
\end{figure}
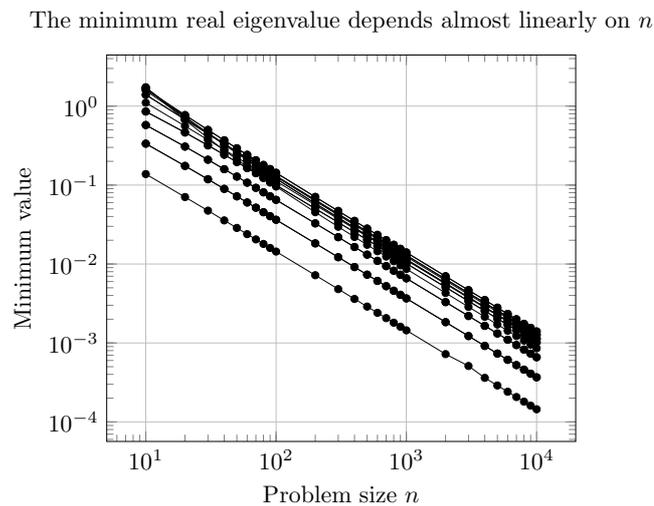




\section{Conclusions}\label{sect:C}

In this article a novel fast method was proposed to analyze and
compute the spectrum of a particular class of hypergeometric
polynomials. Numerical experiments illustrated the viability of the
approach for particular ranges of $\alpha$ and $\kappa$ and some
numerical analysis of the spectrum of the polynomials was presented.


\section*{Acknowledgements}

Nicola Mastronardi is member of the Gruppo Nazionale Calcolo Scientifico-Istituto Nazionale di Alta Matematica (GNCS-INdAM).
The work of Nicola Mastronardi  was partly supported by MIUR,
PROGETTO DI RICERCA DI
RILEVANTE INTERESSE NAZIONALE (PRIN)
20227PCCKZ
"Low--rank Structures and Numerical Methods in Matrix and Tensor
Computations and their Application",
Universit\`a degli Studi di BOLOGNA CUP
J53D23003620006. The work of Marc Van Barel and Raf Vandebril  was partially supported by the Research Council KU Leuven (Belgium), project 
C16/21/002 (Manifactor: Factor Analysis for Maps into Manifolds) and by the Fund for Scientific Research -- Flanders (Belgium), projects G0A9923N (Low rank tensor approximation techniques for up- and downdating of massive online time series clustering) and G0B0123N (Short recurrence relations for rational Krylov and orthogonal rational functions inspired by modified moments).

The authors also wish to thank Prof.\ Zagorodnyuk, for swiftly
responding and pointing to some references of interest for this article.

%



\bibliographystyle{sn-nature}
\bibliography{shortstrings,hyper}

\end{document}